\documentclass[12pt]{amsart}
\usepackage{amsmath}
\usepackage{amssymb}
\usepackage{amsfonts}
\usepackage{amsthm}
\usepackage{array}
\usepackage{mathtools}
\usepackage{amsthm}
\usepackage{amssymb,amsmath}

\theoremstyle{plain} 
\newtheorem{Theorem}{Theorem}
\newtheorem{Lemma}[Theorem]{Lemma}
\newtheorem{proposition}[Theorem]{Proposition}

\newtheorem{Remark}[Theorem]{Remark}

\theoremstyle{remark}

\date{}

\author{\.Ilker Arslan}
\address{Sabanci University,\\ Orhanli, 34956 Tuzla, Istanbul, Turkey}
\email[\.I.~Arslan]{ilkerarslan@sabanciuniv.edu}

\usepackage{lipsum}% http://ctan.org/pkg/lipsum

\begin{document}

\title{Characterization  of the potential smoothness of one-dimensional Dirac operator subject to general boundary conditions and its Riesz basis property}

%\author{\.Ilker Arslan\\Sabanci University,\\ Orhanli, 34956 Tuzla, Istanbul, Turkey}

%\author{\.{I}lker Arslan}
%\address{Sabanci University, Orhanli, 34956 Tuzla, Istanbul, Turkey}

\maketitle

\begin{abstract}
The one-dimensional Dirac operator with periodic potential $V=\begin{pmatrix} 0 & \mathcal{P}(x) \\ \mathcal{Q}(x) & 0 \end{pmatrix}$, where $\mathcal{P},\mathcal{Q}\in L^2([0,\pi])$  subject to periodic, antiperiodic or a general strictly regular boundary condition $(bc)$ has discrete spectrums. It is known that, for large enough $|n|$ in the disc centered at $n$ of radius 1/4, the operator has exactly two (periodic if $n$ is even or antiperiodic if $n$ is odd) eigenvalues $\lambda_n^+$ and $\lambda_n^-$ (counted according to multiplicity) and one eigenvalue $\mu_n^{bc}$ corresponding to the boundary condition $(bc)$. We prove that the smoothness of the potential could be characterized by the decay rate of the sequence $|\delta_n^{bc}|+|\gamma_n|$, where $\delta_n^{bc}=\mu_n^{bc}-\lambda_n^+$ and $\gamma_n=\lambda_n^+-\lambda_n^-.$ Furthermore, it is shown that the Dirac operator with periodic or antiperiodic boundary condition has the Riesz basis property if and only if $\sup\limits_{\gamma_n\neq0} \frac{|\delta_n^{bc}|}{|\gamma_n|}$ is finite.
\end{abstract}

\section{Introduction}

We consider the one-dimensional Dirac operator
\begin{equation}\label{intro00}Ly=i\begin{pmatrix} 1 & 0 \\ 0 & -1
\end{pmatrix}
\frac{dy}{dx}  + \begin{pmatrix} 0 & \mathcal{P}(x) \\ \mathcal{Q}(x) & 0 \end{pmatrix} y, \quad y = \begin{pmatrix} y_1\\y_2
\end{pmatrix},\end{equation}
where $\mathcal{P},\mathcal{Q}\in L^2([0,\pi])$, with periodic, antiperiodic and Dirichlet boundary conditions. We also consider a general boundary condition (bc) given by

\begin{equation}\label{intro11}
  \begin{split}
    y_1(0)+by_1(\pi)+ay_2(0)=0\\
    dy_1(\pi)+cy_2(0)+y_2(\pi)=0,
  \end{split}
\end{equation}
where $a,b,c,d$ are complex numbers subject to the restrictions
\begin{equation}\label{intro22}
b+c=0,\quad ad=1-b^2
\end{equation}
with $ad\neq0$.
It is well-known that if $\mathcal{P},\mathcal{Q}\in L^2([0,\pi])$, $\mathcal{P}=\overline{\mathcal{Q}}$ and we extend $\mathcal{P}$ and $\mathcal{Q}$ as $\pi$-periodic functions on $\mathbb{R}$, then the operator is self-adjoint and has a band-gap structured spectrum of the form

$$Sp(L)=\bigcup_{n=-\infty}^{+\infty}[\lambda_{n-1}^+,\lambda_{n}^-],$$
where

$$\cdots\leq\lambda_{n-1}^+<\lambda_n^-\leq\lambda_{n}^+<\lambda_{n+1}^-\cdots.$$

In addition, Floquet theory shows that the endpoints $\lambda_n^{\pm}$ of these spectral gaps are eigenvalues of the operator (\ref{intro00}) subject to periodic boundary conditions or antiperiodic boundary conditions. Furthermore, the spectrum is discrete for each of the above boundary conditions. Also, for $n\in\mathbb{Z}$ with large enough $|n|$ the disc with center $n$ and radius $1/8$ contains two eigenvalues (counted with multiplicity) $\lambda^+_n$ and $\lambda^-_n$ of periodic (if $n$ is even) or antiperiodic (if $n$ is odd) boundary conditions and as well one eigenvalue $\mu_n^{Dir}$ of Dirichlet boundary condition. There is also one eigenvalue $\mu_n^{bc}=\mu_n$ of the general boundary condition (bc) given above (which will be proven in the first section).

There is a very close relationship  between the smoothness of the potential and the rate of decay of the deviations $|\lambda_n^+-\lambda_n^-|$ and $|\mu_n^{Dir}-\lambda_n^+|$.
The story of the discovery of this relation was initiated by H. Hochstadt \cite{Ho1,Ho2} who considered the (self-adjoint) Hill's operator and proved that the decay rate of the spectral gap $\gamma_n=|\lambda_n^+-\lambda_n^-|$ is $O(1/n^{m-1})$ if the potential has $m$ continuous derivatives. Furthermore, he showed that every finite-zone potential (i.e., $\gamma_n=0$ for all but finitely many $n$) is a $C^{\infty}$-function. Afterwards, some authors \cite{LP,MO1,MT} studied on this relation and showed that if $\gamma_n$ is $O(1/n^k)$ for any $k\in\mathbb{Z}^+$, then the potential is infinitely differentiable. Furthermore, Trubowitz \cite{Tr} showed that the potential is analytic if and only if $\gamma_n$ decays exponentially fast. In the non-selfadjoint case, the potential smoothness still determines the decay rate of $\gamma_n$. However, the decay rate of $\gamma_n$ does not determine the potential smoothness as Gasymov showed \cite{Gas}. In this case, Tkachenko \cite{Tk92,jjsvt,vt} gave the idea to consider $\gamma_n$ together with the deviation $\delta_n^{Dir}=\mu_n^{Dir}-\lambda_n^+$ and obtained characterizations of $C^{\infty}$-smoothness and analyticity of the potential with these deviations $\gamma_n$ and $\delta_n^{Dir}.$ In addition to these developments, Sansuc and Tkachenko \cite{jjsvt2} proved that the potential is in the Sobolev space $H^m$, $m\in\mathbb{N},$ if and only if $\gamma_n$ and $\delta_n^{Dir}$ satisfy
$$\sum(|\gamma_n|^2+|\delta_n^{Dir}|^2)(1+n^{2m})<\infty.$$

The results mentioned above have been obtained by using Inverse Spectral Theory.

Gr\'{e}bert, Kappeler, Djakov and Mityagin studied the relationship between the potential smoothness and the decay rate of spectral gaps for Dirac operators (see \cite{grebert,grebert2,spectra}).

We recall that a characterization of smoothness of a function can be given by weights $\Omega=\Omega(n)_{n\in\mathbb{Z}}$, where the corresponding weighted Sobolev space is 
$$
H(\Omega) = \{v(x) = \sum_{k \in \mathbb{Z}} v_k e^{2ikx}: \quad
\sum_{k \in \mathbb{Z}} |v_k|^2 (\Omega(k))^2 < \infty \}
$$
and the corresponding weighted $\ell^2-$space is
$$\ell^2(\Omega,\mathbb{Z})=\{(x_n)_{n\in\mathbb{Z}}:\sum|x_n|^2(\Omega(n))^2<\infty\}.$$
A weight $\Omega$ is called sub-multiplicative if $\Omega(n+m)\leq\Omega(n)\Omega(m)$ for each $n,m\in\mathbb{Z}.$ It has been proved \cite{spectra,instability,instability2} that for each sub-multiplicative weight $(\Omega(n))_{n\in\mathbb{Z}}$ the following implication holds
$$\mathcal{P},\mathcal{Q}\in H(\Omega)\implies(\gamma_n)\in\ell^2(\Omega,\mathbb{Z}).$$
As mentioned above, the converse does not necessarily hold. However, a converse of this statement was given \cite{spectra,instability} in the self-adjoint case, i.e., when $\mathcal{P}=\overline{\mathcal{Q}}.$ Furthermore, another converse of this statement is shown in terms of sub-exponential weights and a slightly weaker result is obtained in terms of exponential weights in \cite{kst} for Dirac operators with skew-adjoint $L^2-$potentials. Similar results for Schr\"odinger operators were obtained in \cite{KaMi01,DM3,DM5,instability}.

For the non-self-adjoint case, there is a result in \cite{instability} as follows: Let us put
$$\Delta_n=|\lambda^+_n-\lambda^-_n|+|\lambda^+_n-\mu_n^{Dir}|,$$
then for each sub-multiplicative weight $\Omega$
$$\mathcal{P},\mathcal{Q}\in H(\Omega)\implies(\Delta_n)_{n\in\mathbb{Z}}\in\ell^2(\Omega).$$
Moreover, if $\Omega=\Omega(n)_{n\in\mathbb{Z}}$ is a sub-multiplicative weight such that $\log\Omega(n)/n\searrow 0$, then
$$(\Delta_n)_{n\in\mathbb{Z}}\in\ell^2(\Omega)\implies\mathcal{P},\mathcal{Q}\in H(\Omega)$$
and if $\lim\limits_{n\rightarrow\infty}\log\Omega(n)/n>0$, then
$$(\Delta_n)_{n\in\mathbb{Z}}\in\ell^2(\Omega)\implies\exists\;\epsilon>0 : \mathcal{P},\mathcal{Q}\in H(e^{\epsilon|n|}).$$
The proofs are constructed by means of the matrix
$$\begin{pmatrix} \alpha_n(z) & \beta^+_n(z) \\ \beta^-_n(z) &
\alpha_n(z)\end{pmatrix},$$
which has the very important property that a number $\lambda=n+z$ with $|z|<1/2$ is a periodic (if $n$ is even) or antiperiodic (if $n$ is odd) eigenvalue if and only if $z$ is an eigenvalue of the matrix (see Lemma 21, \cite{instability}). The four entries of the matrix depend analytically on $z$ and $V$. They are given explicitly in terms of the Fourier coefficients of $V.$
The deviations $|\gamma_n|+|\delta_n^{Dir}|$ are estimated (see Theorem 66 in \cite{instability}) by the functionals $\beta_n^{\mp}(z)$ as follows
$$\frac{1}{144}(|\beta_n^-(z_n^*)|+|\beta_n^+(z_n^*)|)\leq |\gamma_n|+|\delta_n^{Dir}|\leq54(|\beta_n^-(z_n^*)|+|\beta_n^+(z_n^*)|),$$
where $z_n^*=(\lambda_n^++\lambda_n^-)/2-n.$ This shows the significance of these functionals by means of their asymptotic equivalence with the sequence $|\gamma_n|+|\delta_n^{Dir}|.$

The functionals $\alpha_n(z)$ and $\beta_n^{\mp}(z)$ are also crucial in analysing the Riesz basis property of the Dirac operator. P. Djakov and B. Mityagin \cite{criteria} have proved that the following three claims are equivalent:

(a) The Dirac operator $L$ given by (\ref{intro00}) with a potential $V$ in $L^2([0,\pi])\times L^2([0,\pi])$ subject to periodic or antiperiodic boundary conditions has the Riesz basis property.

(b) $0<\liminf\limits_{\gamma_n\neq0}\frac{|\beta_n^-(z_n^*)|}{|\beta_n^+(z_n^*)|}\quad\text{and}\quad\limsup\limits_{\gamma_n\neq0}\frac{|\beta_n^-(z_n^*)|}{|\beta_n^+(z_n^*)|}<\infty.$

(c) $\sup\limits_{\gamma_n\neq0}\frac{|\delta_n^{Dir}|}{|\gamma_n|}<\infty.$

Similar results concerning Riesz basis property are known for Schr\"odinger operators (see \cite{criteria,GT11,AB01}).

In this paper are obtained new results on potential smoothness and Riesz basis property of one-dimensional Dirac operators. The following theorems give the main results.

\begin{Theorem}\label{intro333}
If \ $V \in L^2([0,\pi])\times L^2([0,\pi])$, then 
$$V\in H(\Omega)\implies(\Delta_n^{bc})_{n\in\mathbb{Z}}\in\ell^2(\Omega)$$
for each submultiplicative weight $\Omega$, where
$$\Delta_n^{bc}=|\lambda^+_n-\lambda^-_n|+|\lambda^+_n-\mu_n^{bc}|.$$

Conversely, if $\Omega=(\Omega(n))_{n\in\mathbb{Z}}$ is a submultiplicative weight such that
$\frac{\log\Omega(n)}{n}\searrow 0$, then 
$$(\Delta_n^{bc})_{n\in\mathbb{Z}}\in\ell^2(\Omega)\implies V\in H(\Omega).$$

Furthermore, if $\Omega$ is a submultiplicative weight such that $\lim\limits_{n\rightarrow\infty}\frac{\log\Omega(n)}{n}>0,$ then
$$(\Delta_n^{bc})_{n\in\mathbb{Z}}\in\ell^2(\Omega)\implies\exists\;\epsilon>0 : \ V\in H(e^{\epsilon|n|}).$$
\end{Theorem}

\begin{Theorem}\label{intro4444}
If $\ V\in L^2([0,\pi])\times L^2([0,\pi])$, then the Dirac operator (\ref{intro00}) with periodic or antiperiodic boundary conditions has the Riesz basis property if and only if
$$\sup\limits_{\gamma_n\neq0}\frac{|\delta_n^{bc}|}{|\gamma_n|}<\infty$$
holds, where $\delta_n^{bc}=\lambda^+_n-\mu_n^{bc}.$
\end{Theorem}

Primarily, the following theorem is proven as a generalization of Theorem 66 in \cite{instability}.

\begin{Theorem}\label{intro55555}
For $n\in\mathbb{Z}$ with large enough $|n|$, there are constants $K_1>0$ and $K_2>0,$ such that
$$K_1(|\beta_n^-(z_n^*)|+|\beta_n^+(z_n^*)|)\leq|\lambda^+_n-\lambda_n^-|+|\mu_n^{bc}-\lambda_n^+|\leq K_2(|\beta_n^-(z_n^*)|+|\beta_n^+(z_n^*)|).$$
\end{Theorem}

Theorem \ref{intro333} and Theorem \ref{intro4444} are not proven directly. However, their proofs are reduced to the proofs of Theorem 66 in \cite{instability} and Theorem 24 in \cite{criteria}, respectively, in which we make use of Theorem \ref{intro55555} that gives the asymptotic equivalence of $|\beta_n^-(z_n^*)|+|\beta_n^+(z_n^*)|$ and $|\lambda^+_n-\lambda_n^-|+|\mu_n^{bc}-\lambda_n^+|$.

Some of the necessary estimates in the proof of Theorem \ref{intro55555} are done by a method discovered by Ahmet Batal \cite {ahmet, ahmetarxive} in the context of Schr\"odinger operators.

\section{Preliminaries}

A general boundary condition for the operator $L$ is given by
\begin{eqnarray*}
a_1y_1(0)+b_1y_1(\pi)+a_2y_2(0)+b_2y_2(\pi)=0,\\
c_1y_1(0)+d_1y_1(\pi)+c_2y_2(0)+d_2y_2(\pi)=0,
\end{eqnarray*}
where $a_i,b_i,c_i,d_i$ $(i=1,2)$ are complex numbers. 

Let $A_{ij}$ denote the square matrix whose first and second columns are the $i^{th}$ and $j^{th}$ columns of the matrix
$$\begin{bmatrix}a_1&b_1&a_2&b_2\\c_1&d_1&c_2&d_2\end{bmatrix}$$
respectively and let $|A_{ij}|$ be the determinant of $A_{ij}$. If $|A_{14}|\neq0$
and $|A_{23}|\neq0$, then we say that the boundary condition given
above is $regular$, if additionally
$(|A_{13}|+|A_{24}|)^2\neq4|A_{14}||A_{23}|$ holds, it is called
$strictly\ regular$.

%% DESCRIPTION OF THE SPECIAL BOUNDARY CONDITIONS...

Description of a family of special boundary conditions: Consider
matrices of the form
\begin{equation}\label{general}
A=\begin{bmatrix}1&b&a&0\\0&d&c&1\end{bmatrix},
\end{equation}
where $a,b,c,d$ are complex numbers.
For every such matrix, the corresponding boundary condition $bc(A)$ is given by

\begin{equation}\label{ilker2}
  \begin{split}
    y_1(0)+by_1(\pi)+ay_2(0)=0,\\
    dy_1(\pi)+cy_2(0)+y_2(\pi)=0.
  \end{split}
\end{equation}

We consider the family of all such boundary conditions that satisfy
also
\begin{equation}\label{ilker3}
b+c=0,\quad ad=1-b^2
\end{equation}
with restriction $ad\neq0.$ 

From now on, we denote by $bc$ the boundary conditions given by (\ref{ilker2}) with restrictions (\ref{ilker3}). Observe that $bc$ is strictly regular.

\section{Localization of the spectra}
We give the localization of the spectra of Dirac operator subject to three types of boundary conditions which are the general boundary conditions defined by (\ref{ilker2}) and (\ref{ilker3}), periodic and antiperiodic boundary conditions defined as follows
 \begin{align*}
&\text{Periodic}\;(bc=Per^+): \quad\quad\quad y(0) = y (\pi), \;\;\;\; \text{i.e. }y_1(0)=y_1(\pi) \text{ and }y_2(0)=y_2(\pi)  ;\nonumber \\
&\text{Antiperiodic}\;(bc=Per^-): \quad y(0) = -y (\pi), \;\; \text{i.e. }y_1(0)=-y_1(\pi) \text{ and }y_2(0)=-y_2(\pi).\nonumber
\end{align*}

We denote by $L$ the Dirac operator subject to $Per^{\pm}$ and denote by $L_{bc}$ the Dirac operator with general boundary conditions $bc=bc(A)$, where $A$ is given by (\ref{general}), (\ref{ilker2}) and (\ref{ilker3}). We also denote by $L^0$ and $L_{bc}^0$ the corresponding free Dirac operators subject to $Per^{\pm}$ and $(bc)$ when $V=0.$

We consider $L$ in the domain $dom(L)$, which consists of all absolutely continuous functions $y$ such that $y'\in L^2([0,\pi])\times L^2([0,\pi])$ and $y$ satisfy $(Per^{\pm})$. Furthermore, we consider $L_{bc}$ in the domain $dom(L_{bc})$, which consists of all absolutely continuous functions $y$ such that $y'\in L^2([0,\pi])\times L^2([0,\pi])$ and $y$ satisfies $bc.$

%The following theorem is about the localization of the eigenvalues of the Dirac operator $L_{bc}=L^0_{bc}+V$ under the given general boundary conditions defined by (\ref{general}).

\begin{Theorem}\label{localizationgeneral}
%Let A be a matrix of the form (\ref{general}), and let $bc=bc(A)$ be
%the corresponding boundary condition. If $A$ satisfies
%(\ref{ilker3}), then 
The spectrum of the free operator $L_{bc}^0$ is
given by $sp(L_{bc}^0)=\mathbb{Z}.$ Moreover, for $n\in\mathbb{Z}$ with large enough $|n|$, the disc $D_n=\{z\in \mathbb{C}:|z-n|<1/2\}$ contains one simple
eigenvalue $\mu_n=\mu_n(bc)$ of the operator $L_{bc}.$
\end{Theorem}

\begin{proof}
First, we consider the equation

\begin{equation*}
L_{bc}^0\ y=\lambda y,\quad y = \begin{pmatrix} y_1\\y_2
\end{pmatrix}.
\end{equation*}

Its solution is of the form
\begin{equation*}
y=\begin{pmatrix} \xi e^{-i\lambda x} \\ \zeta e^{i\lambda x}\end{pmatrix}.
\end{equation*}

To satisfy the general boundary conditions given by (\ref{ilker2}), $(\xi,\zeta)$ must be a solution of the linear system 

\begin{equation}\label{ariza}
\begin{bmatrix}1+bz^{-1}&a\\dz^{-1}&c+z\end{bmatrix}\begin{bmatrix}\xi\\ \zeta\end{bmatrix}=\begin{bmatrix}0\\ 0\end{bmatrix},
\end{equation}
where $z=e^{i\pi\lambda}.$

In order to have a non-zero solution $(\xi,\zeta)$ the determinant of the matrix in (\ref{ariza})
has to be zero. So, we obtain
\begin{equation*}
z^2+(b+c)z+bc-ad=0.
\end{equation*}

Together with the restrictions (\ref{ilker3}), we have $z^2=1$ which gives $z=\mp1.$ Hence, one may conclude that $sp(L_{bc}^0)=\mathbb{Z}$ since the only solutions to the equation $\mp1=e^{i\pi\lambda}$ are integers. The second result comes from Theorem 5.3 of \cite{unconditional} since $bc$ is strictly regular.

\end{proof}

\section{The eigenvectors of the free operators}

%% RIESZ BASIS TO BE GIVEN FOR THE GENERAL BOUNDARY CONDITIONS

The adjoint boundary conditions $(bc^*)$ are given by the matrix (see Lemma 3.4 in \cite{unconditional})

\begin{equation}\label{adjointbc}
A^*=\begin{bmatrix}1&\tilde{b}&\tilde{a}&0\\0&\tilde{d}&\tilde{c}&1\end{bmatrix}=\begin{bmatrix}1&-\overline{c}&\overline{d}&0\\0&\overline{a}&-\overline{b}&1\end{bmatrix},
\end{equation}
due to the restrictions (\ref{ilker3}), where
\begin{equation*}
\begin{bmatrix}\tilde{b}&\tilde{a}\\\tilde{d}&\tilde{c}\end{bmatrix}=\bigg(\begin{bmatrix}b&a\\d&c\end{bmatrix}^{-1}\bigg)^*.
\end{equation*}

%After a calculation by using (\ref{ilker3}) we find that

%$$\tilde{a}=\overline{d},\quad\tilde{b}=-\overline{c},\quad\tilde{c}=-\overline{b},\quad\tilde{d}=\overline{a}.$$

%Therefore, the adjoint boundary conditions are given by the matrix

%\begin{equation}\label{adjointbc}
%A^*=\begin{bmatrix}1&-\overline{c}&\overline{d}&0\\0&\overline{a}&-\overline{b}&1\end{bmatrix}.
%\end{equation}

Furthermore, we observe that the adjoint boundary conditions are also in the family of general boundary conditions given by (\ref{ilker2}) and (\ref{ilker3}) since $$\tilde{b}+\tilde{c}=0\quad\text{and}\quad\tilde{a}\tilde{d}=1-\tilde{b}^2.$$ Hence, by Theorem \ref{localizationgeneral} we get that the eigenvalues of $(L_{bc}^0)^*$ are also integers. 
%As in the proof of Theorem \ref{localizationgeneral}, we see that every eigenvector is of the form
%\begin{equation*}
%y=\begin{pmatrix} \xi e^{-i\lambda x} \\ \zeta e^{i\lambda x}\end{pmatrix},
%\end{equation*}
%with $\xi,\zeta\in\mathbb{C}$. Because the eigenvectors satisfy the boundary conditions (\ref{adjointbc}) with (\ref{hhhh}), we obtain
%\begin{eqnarray}\label{ariza33}
%\xi(1-\overline{c}z^{-1})+\zeta \overline{d}& = & 0, \\
%\label{ariza234}\xi \overline{a}z^{-1}+\zeta(-\overline{b}+z) & = & 0,
%\end{eqnarray}
%where $z=e^{i\pi\lambda}.$
%Then, for $n\in 2\mathbb{Z}$, the corresponding eigenfunction is

%\begin{equation}
%\label{forme1}
%\begin{pmatrix} (1-\overline{b})e^{-inx}\\-\overline{a}e^{inx}\end{pmatrix}
%\end{equation}
%and for $n\in 2\mathbb{Z}+1$, the corresponding eigenfunction is

%\begin{equation}
%\label{forme2}
%\begin{pmatrix} (1+\overline{b})e^{-inx}\\-\overline{a}e^{inx}\end{pmatrix}.
%\end{equation}

Hence, all eigenfunctions corresponding to an eigenvalue $n\in\mathbb{Z}$ are

\begin{equation}
\label{ }
\begin{pmatrix}
A_ne^{-inx}\\B_ne^{inx}
\end{pmatrix},
\end{equation}
where

\begin{equation}\label{onedimension1}
A_n=\frac{1-(-1)^n\overline{b}}{\sqrt{|a|^2+|1-(-1)^n\overline{b}|^2}}\ ,\quad B_n=\frac{-\overline{a}}{\sqrt{|a|^2+|1-(-1)^n \overline{b}|^2}}
\end{equation}
%and
%\begin{equation}\label{onedimension2}
%B_n=\frac{-\overline{a}}{\sqrt{|a|^2+|1-(-1)^n \overline{b}|^2}}\ ,
%\end{equation}
so that $|A_n|^2+|B_n|^2=1.$

Since the adjoint boundary conditions $bc^*$ are strictly regular,
$$\lim_{n\to\infty}\|P_{n,bc^*}-P_{n,bc^*}^0\|=0,$$
due to Theorem 6.1 in \cite{unconditional}, where 
\begin{equation*}
P_{n,bc^*}=\int_{\partial D_n}(\lambda-L_{bc^*})^{-1}d\lambda, \quad P_{n,bc^*}^0=\int_{\partial D_n}(\lambda-L_{bc^*}^0)^{-1}d\lambda
%\end{equation*}
%and
%\begin{equation*}
%P_{n,bc^*}^0=\int_{\partial D_n}(\lambda-L_{bc^*}^0)^{-1}d\lambda.
\end{equation*}
are the Cauchy-Riesz projections, $D_n=\{z\in \mathbb{C}:|z-n|<1/2\}$ and $\partial D_n$ is the boundary of $D_n$. Notice that $L_{bc^*}^0=(L_{bc}^0)^*$ and $L_{bc^*}=(L_{bc})^*$ (see Lemma 3.4 in \cite{unconditional}).

Similarly, we have the Cauchy-Riesz projections $P_n$ and $P_n^0$ associated with Dirac operator $L$ with periodic boundary conditions if $n$ is even and antiperiodic boundary conditions if $n$ is odd, where $dim(P_n)=dim(P_n^0)=2$ due to Theorem 18 in \cite{instability}. Furthermore,
$$\lim_{n\to\infty}\|P_n-P_n^0\|=0$$
by Proposition 19 in \cite{instability} and for large enough $|n|$, the operator $L=L^0+V$ has two eigenvalues $\lambda_n^+$ and $\lambda_n^-$ (which are periodic for even $n$ and antiperiodic for odd $n$) such that $|\lambda_n^{\pm}-n|\leq1/4$ as a result of Theorem 17 in \cite{instability}.

Furthermore, the spectrum of the free operator $L^0$ subject to periodic boundary condition is $2\mathbb{Z}$ and each $n\in2\mathbb{Z}$ is a double eigenvalue and the corresponding eigenvectors are
\begin{equation}\label{offf1}e_n^1(x)=\begin{pmatrix} e^{-inx}\\0
\end{pmatrix},\quad
%\end{equation}
%and
%\begin{equation}\label{offf2}
e_n^2(x)=\begin{pmatrix} 0\\e^{inx}
\end{pmatrix}.
\end{equation}

Similarly, the spectrum of the free operator $L^0$ subject to antiperiodic boundary condition is $2\mathbb{Z}+1$ and each $n\in2\mathbb{Z}+1$ is a double eigenvalue and the corresponding eigenvectors are obtained by the same formulae (\ref{offf1}). So, we may write $E_n^0=Span\{e_n^1,e_n^2\}$ for all $n\in\mathbb{Z}.$ Moreover, we may also write
$E_n=\text{Range}(P_n)$ and $E_n^0=\text{Range}(P_n^0)$
for the eigenspaces of the operators $L$ and $L^0,$ respectively.

\section{Estimates for $|\mu_n-\lambda_n^+|$}

The Dirac operator $L=L^0+V$ has two eigenvalues $\lambda^+_n$ and $\lambda^-_n$ (periodic for even $n$ and antiperiodic for odd $n$) in the disc centered at $n\in \mathbb{Z}$ of radius 1/4 for large enough $|n|$ (Theorem 17 and Theorem 18 in \cite{instability}). We denote by $\lambda_n^+$ the eigenvalue with larger real part or the one with larger imaginary part if the real parts are equal and we put $\gamma_n=\lambda^+_n-\lambda^-_n$.

From Lemma 59 in \cite{instability}, for sufficiently large $|n|$, there is a pair of vectors $f_n,\varphi_n \in E_n$ such that
\begin{enumerate}
  \item $\|f_n\|=1,\|\varphi_n\|=1, \langle f_n,\varphi_n \rangle=0$
  \item $Lf_n=\lambda_n^+f_n$
  \item $L\varphi_n=\lambda_n^+\varphi_n -\gamma_n\varphi_n+\xi_n f_n$
\end{enumerate}
for some sequence $(\xi_n) \in \mathbb{C}.$
%and
%\begin{equation*}
%|\xi_n|\leq 4|\gamma_n|+2\|(z_n^+-S(\lambda_n^+))P_n\varphi_n \|
%\end{equation*}
%and
%\begin{equation*}
%\|(z_n^+-S(\lambda_n^+))P_n\varphi_n\|\leq2(|\xi_n|+|\gamma_n|),
%\end{equation*}
%where $S(\lambda_n^+):E_n^0\rightarrow E_n^0$ is the operator defined in Lemma 21 of \cite{instability}, and $z_n^+=\lambda_n^+-n.$

Now, let $\ell_0$ and $\ell_1$ be the functionals from
$C([0,\pi])\times C([0,\pi])$ to $\mathbb{C}$, defined as

\begin{eqnarray*}
\ell_0(s)&=&s_1(0)+bs_1(\pi)+as_2(0),\\
\ell_1(s)&=&ds_1(\pi)+cs_2(0)+s_2(\pi),
\end{eqnarray*}
where
$$s(x)=\begin{pmatrix}s_1(x)\\s_2(x)\end{pmatrix}.$$

We start with a crucial lemma which gives us the restrictions on those regular boundary conditions by which Dirichlet condition in \cite{instability} could be replaced. Furthermore, this will lead to an equation that determines the way we estimate $|\mu_n-\lambda_n^+|.$ 

\begin{Lemma}\label{existence}
For large enough $|n|$ there is vector $G_n\in E_n$ of the form
\begin{equation*}
G_n=s_nf_n+t_n\varphi_n,\quad \|G_n\|=|s_n|^2+|t_n|^2=1
\end{equation*}
such that
\begin{equation*}
\ell_0(G_n)=0,\quad \ell_1(G_n)=0
\end{equation*}
if the general boundary conditions (\ref{ilker2}) satisfy (\ref{ilker3}).
\end{Lemma}

\begin{proof}
It will be enough to prove that the system of linear equations
\begin{align}\label{newnew}
\begin{split}
\ell_0(s_nf_n+t_n\varphi_n)  =  0  \\
\ell_1(s_nf_n+t_n\varphi_n)  =  0
\end{split}
\end{align}
has a non-trivial solution if $b+c=0$ and $ad=1-b^2$ hold.

Now, the system can be written as follows
\begin{eqnarray*}
s_nf_n^1(0)+t_n\varphi_n^1(0)+b(s_nf_n^1(\pi)+t_n\varphi_n^1(\pi))+a(s_nf_n^2(0)+t_n\varphi_n^2(0)) & = & 0 \\
d(s_nf_n^1(\pi)+t_n\varphi_n^1(\pi))+c(s_nf_n^2(0)+t_n\varphi_n^2(0))+s_nf_n^2(\pi)+t_n\varphi_n^2(\pi) & = & 0, 
\end{eqnarray*}
where
$$f_n=\begin{pmatrix}f_n^1\\f_n^2\end{pmatrix}\text{ and }\quad \varphi_n=\begin{pmatrix}\varphi_n^1\\\varphi_n^2\end{pmatrix}.$$

Then, since $f_n$ and $\varphi_n$ satisfy periodic boundary conditions, we can reduce the above system to

\begin{equation}\label{farte}
\begin{pmatrix}
(1+b)f_n^1(0)+af_n^2(0)  & (1+b)\varphi_n^1(0)+a\varphi_n^2(0)\\
df_n^1(0)+(1+c)f_n^2(0)  & d\varphi_n^1(0)+(1+c)\varphi_n^2(0)      
\end{pmatrix}
\begin{pmatrix}
      s_n  \\
      t_n  
\end{pmatrix}=
\begin{pmatrix}
      0  \\
      0  
\end{pmatrix}.
\end{equation}
The restrictions (\ref{ilker3}) leads to the fact that $(1+b)(1+c)-ad=0$. Then,

\begin{equation*}
[ (1+b)(1+c)-ad].[f_n^1(0)\varphi_n^2(0)-f_n^2(0)\varphi_n^1(0)]=0,
\end{equation*}
which is the determinant of the matrix corresponding to the system (\ref{farte}).
In a similar way from the antiperiodic case of $f$ and $\varphi$, we get

\begin{equation*}
[ (1-b)(1-c)-ad].[f_n^1(0)\varphi_n^2(0)-f_n^2(0)\varphi_n^1(0)]=0,
\end{equation*}
which is the corresponding determinant. Hence, if $b+c=0$ and $ad=1-b^2$, then we have a non-trivial solution to the system (\ref{newnew}).
\end{proof}

\begin{Remark}The converse of Lemma \ref{existence} is also true, but we will not use this fact. In order to prove it, one may check that 
$$f_n^1(0)\varphi_n^2(0)-f_n^2(0)\varphi_n^1(0)\neq0$$
for large enough $|n|$ by making use of Remark (\ref{remarkx}) and (\ref{test1}).
\end{Remark}
As seen by the proof of the previous lemma, we can write $G_n$ as
\begin{equation}
\label{linear1}
G_n=\tau_n(\ell_0(\varphi_n)f_n-\ell_0(f_n)\varphi_n),
\end{equation}
%or 
%\begin{equation}
%\label{linear2}
%G=\tau(G)(\ell_1(\varphi)f-\ell_1(f)\varphi)
%\end{equation}
where
$$\tau_n=\frac{1}{\sqrt{|\ell_0(\varphi_n)|^2+|\ell_0(f_n)|^2}}\cdot$$

We also write $G_n=s_nf_n+t_n\varphi_n,$ where $s_n=\tau_n\ell_0(\varphi_n)$ and $t_n=-\tau_n\ell_0(f_n).$

Now, since $G_n$ is in the domain of $L_{bc}$ and $L$, we can continue to write

\begin{eqnarray*}
L_{bc}G_n & = & LG_n=s_n.Lf_n+t_n.L\varphi_n=s_n\lambda_n^+f_n+t_n(\lambda_n^+\varphi_n-\gamma_n\varphi_n+\xi_n f_n) \\
& = & \lambda_n^+(s_nf_n+t_n\varphi_n)+t_n(\xi_n f_n-\gamma_n\varphi_n)=\lambda_n^+G_n+t_n(\xi_n f_n-\gamma_n\varphi_n).
\end{eqnarray*}

So, we have

\begin{equation}
\label{main}
LG_n =\lambda_n^+G_n+t_n(\xi_n f_n-\gamma_n\varphi_n).
\end{equation}

Let $\tilde{g}_n$ be a unit eigenvector of the adjoint operator $(L_{bc})^*$ corresponding to the eigenvalue $\overline{\mu_n},$ where $\mu_n$ is the eigenvalue of $L_{bc}$ in a circle with center $n$ and radius $1/4.$

Taking inner products of both sides of the equation (\ref{main}) by $\tilde{g}_n$ we obtain

\begin{equation}\label{six}
\langle LG_n,\tilde{g}_n\rangle=\lambda_n^+\langle
G_n,\tilde{g}_n\rangle+t_n(\xi_n \langle
f_n,\tilde{g}_n\rangle-\gamma_n\langle\varphi_n,\tilde{g}_n\rangle).
\end{equation}

We also have
\begin{equation}\label{seven}
\langle LG_n,\tilde{g}_n\rangle=\langle L_{bc}G_n,\tilde{g}_n\rangle=\langle G_n,(L_{bc})^*\tilde{g}_n\rangle=\langle G_n,\overline{\mu_n}\tilde{g}_n\rangle=\mu_n\langle G_n,\tilde{g}_n\rangle.
\end{equation}
The equality of (\ref{six}) and (\ref{seven}) leads to
\begin{equation}\label{ilker6}
(\mu_n-\lambda_n^+)\langle G_n,\tilde{g}_n\rangle=t_n(\xi_n \langle
f_n,\tilde{g}_n\rangle-\gamma_n\langle\varphi_n,\tilde{g}_n\rangle).
\end{equation}

The equation (\ref{ilker6}) is important that our proof of the estimation for $|\mu_n-\lambda_n^+|$ will be based on the approximations for each remaining term in (\ref{ilker6}).

Note that for large enough $|n|$, since $f_n\in E_n$ and $P_n$ is a projection onto $E_n$ we have $P_nf_n=f_n.$ So,
$$\|P_n^0f_n\|=\|P_nf_n-(P_n-P_n^0)f_n\|\geq\|f_n\|-\|P_n-P_n^0\|=1-\|P_n-P_n^0\|.$$
Since $\|P_n-P_n^0\|$ is sufficiently small we have that $P_n^0f_n\neq0.$

Now, we introduce notations for the projections of the eigenvectors of Dirac operator under periodic (or antiperiodic) boundary conditions and adjoint boundary conditions $(bc^*)$ given by (\ref{adjointbc}):

\begin{equation*}
f_n^0=\frac{P_n^0f_n}{\|P_n^0f_n\|},\quad\varphi_n^0=\frac{P_n^0\varphi_n}{\|P_n^0\varphi_n\|},\quad\tilde{g}_n^0=\frac{P_{n,bc^*}^0\tilde{g}_n}{\|P_{n,bc^*}^0\tilde{g}_n\|}\cdot
\end{equation*}

Also, we may put
\begin{equation*}
f_n^0=f_{n,1}^0e_n^1+f_{n,2}^0e_n^2\ , \quad \varphi_n^0=\varphi_{n,1}^0e_n^1+\varphi_{n,2}^0e_n^2\ .
\end{equation*}

\begin{Lemma}\label{remark1} In the above notations, for large enough $|n|$ we have
$$\|g_n-\tilde{g}_n^0\|\leq2\|P_{n,bc^*}-P_{n,bc^*}^0\|,\quad\|f_n-f_n^0\|\leq2\|P_n-P_n^0\|, \quad\|\varphi_n-\varphi_n^0\|\leq2\|P_n-P_n^0\|.$$
\end{Lemma}
\begin{proof}
Observe that $P_{n,bc^*}\tilde{g}_n=\tilde{g}_n$ since $P_{n,bc^*}$ is a projection onto the one-dimensional eigenspace generated by $\tilde{g}_n.$
Now, we estimate $\|P_{n,bc^*}^0\tilde{g}_n\|$ as
\begin{align*}
\|P_{n,bc^*}^0\tilde{g}_n\|&=\|P_{n,bc^*}^0\tilde{g}_n-P_{n,bc^*}\tilde{g}_n+P_{n,bc^*}\tilde{g}_n\|\\&\geq\|P_{n,bc^*}\tilde{g}_n\|-\|(P_{n,bc^*}-P_{n,bc^*}^0)\tilde{g}_n\|\\
&\geq\|\tilde{g}_n\|-\|P_{n,bc^*}-P_{n,bc^*}^0\|\\
%&\geq1-\kappa_n\\
\end{align*}
and
\begin{align*}
\|P_{n,bc^*}^0\tilde{g}_n\|&=\|P_{n,bc^*}^0\tilde{g}_n-P_{n,bc^*}\tilde{g}_n+P_{n,bc^*}\tilde{g}_n\|\\&\leq\|(P_{n,bc^*}-P_{n,bc^*}^0)\tilde{g}_n\|+\|P_{n,bc^*}\tilde{g}_n\|\\
&\leq\|P_{n,bc^*}-P_{n,bc^*}^0\|+\|\tilde{g}_n\|\\
%&\leq1+\kappa_n.
\end{align*}
We get from the above inequalities that 
$|\|P_{n,bc^*}^0\tilde{g}_n\|-1|\leq\|P_{n,bc^*}-P_{n,bc^*}^0\|$ since $\|\tilde{g}_n\|=1.$ Therefore,
\begin{align*}
\|\tilde{g}_n-g_n^0\|&\leq\|\tilde{g}_n-P_{n,bc^*}^0\tilde{g}_n\|+\|P_{n,bc^*}^0\tilde{g}_n-g_n^0\|\\
&=\|P_{n,bc^*}\tilde{g}_n-P_{n,bc^*}^0\tilde{g}_n\|+\|\|P_{n,bc^*}^0\tilde{g}_n\|.g_n^0-g_n^0\|\\
&\leq\|(P_{n,bc^*}-P_{n,bc^*}^0)\tilde{g}_n\|+|\|P_{n,bc^*}^0\tilde{g}_n\|-1|.\|g_n^0\|\\
&\leq2\|P_{n,bc^*}-P_{n,bc^*}^0\|.
\end{align*}
Similarly we get the other inequalities.
\end{proof}

\begin{Lemma}\label{summable}
$\sum_{n\in\mathbb{Z}}|\lambda_n^{\mp}-n|^2<\infty$ and $\sum_{n\in\mathbb{Z}}|\lambda_n^+-\lambda_n^-|^2<\infty.$
\end{Lemma}
\begin{proof}
Our proof is similar to the proof of Theorem 6.5 in \cite{unconditional}. First, we prove that $\sum_{n\in\mathbb{Z}}|\lambda_n^+-n|^2<\infty$. Consider the eigenfunctions $e_n^1$ and $e_n^2$ of the free operator with periodic condition and antiperiodic condition. Now, we have
$$\lambda_n^+\langle f_n,e_n^1\rangle=\langle Lf_n,e_n^1\rangle=\langle L^0f_n,e_n^1\rangle+\langle Vf_n,e_n^1\rangle$$
and recalling that $L^0$ is self-adjoint we obtain
$$\langle L^0f_n,e_n^1\rangle=\langle f_n,L^0e_n^1\rangle=\langle f_n,ne_n^1\rangle=n\langle f_n,e_n^1\rangle.$$
From these two equalities, we can write
$$(\lambda_n^+-n)\langle f_n,e_n^1\rangle=\langle Vf_n,e_n^1\rangle.$$
In a similar way, we get
$$(\lambda_n^+-n)\langle f_n,e_n^2\rangle=\langle Vf_n,e_n^2\rangle.$$
The last two equalities lead to
\begin{equation}
\label{sthng}
|\lambda_n^+-n|^2(|\langle f_n,e_n^1\rangle|^2+|\langle f_n,e_n^2\rangle|^2)=|\langle Vf_n,e_n^1\rangle|^2+|\langle Vf_n,e_n^2\rangle|^2.
\end{equation}
Now,
$$\langle f_n,e_n^1\rangle=\langle f_n-f_n^0,e_n^1\rangle+\langle f_n^0,e_n^1\rangle.$$
Since
$$|\langle f_n-f_n^0,e_n^1\rangle|\leq\|f_n-f_n^0\|\leq2\|P_n-P_n^0\|,$$
by Lemma \ref{remark1}, we have
$$\langle f_n,e_n^1\rangle=f_{n,1}^0+O(\|P_n-P_n^0\|).$$
A similar argument gives
$$\langle f_n,e_n^2\rangle=f_{n,2}^0+O(\|P_n-P_n^0\|).$$
We obtain that 

$$|\langle f_n,e_n^1\rangle|^2+|\langle f_n,e_n^2\rangle|^2\rightarrow |f_{n,1}^0|^2+|f_{n,2}^0|^2=1$$
as $n\rightarrow\infty.$

%\begin{Theorem}
%\label{summable}
%In the above notations, we have 
%$$\sum_{n\in \mathbb{Z}}|\lambda_n^+-n|^2<\infty,$$
%$$\sum_{n\in \mathbb{Z}}|\lambda_n^--n|^2<\infty$$
%and
% $$\sum_{n\in \mathbb{Z}}|\lambda_n^+-\lambda_n^-|^2<\infty.$$
%\end{Theorem}

%\begin{proof}
%The proof will be done just for the first series because the second can be obtained in exactly the same %way replacing $f$ by a unit eigenfunction $h$ for $\lambda_n^-$ so that it satisfies $Lh=\lambda_n^-h$. %Also, the third result comes from the previous two results by triangle inequality in $\ell_2(\mathbb{Z}).$

%$$\langle f,e_n^1\rangle=\langle f-f^0,e_n^1\rangle+\langle f^0,e_n^1\rangle$$
%Since
%$$|\langle f-f^0,e_n^1\rangle|\leq\|f-f^0\|\leq\|P-P^0\|\leq\kappa_n,$$
%So, we have
%$$\langle f,e_n^1\rangle=f_1^0+O(\kappa_n)$$
%Following a similar argument gives
%$$\langle f,e_n^2\rangle=f_2^0+O(\kappa_n).$$
%We obtain that 

%$$|\langle f,e_n^1\rangle|^2+|\langle f,e_n^2\rangle|^2\rightarrow |f_1^0|^2+|f_2^0|^2=1$$
%as $n\rightarrow\infty.$

Now, if we consider the equation (\ref{sthng}), it is clear that for large enough $|n|$
\begin{equation}
\label{str}
|\lambda_n^+-n|^2\leq 2(|\langle Vf_n,e_n^1\rangle|^2+|\langle Vf_n,e_n^2\rangle|^2).
\end{equation}

We obtain an estimation for the first term of the right hand side of the above inequality as follows
\begin{eqnarray*}
\langle Vf_n,e_n^1\rangle&=&\langle V(f_{n,2}^0e_n^2),e_n^1\rangle+\langle V(f_n-f_{n,2}^0e_n^2),e_n^1\rangle \\&=&\frac{f_{n,2}^0}{\pi}\int_{0}^{\pi}\mathcal{P}(x)e^{-2inx}dx+\langle f_n-f_n^0,V^{*}e_n^1\rangle+\langle f_n^0-f_{n,2}^0e_n^2,V^{*}e_n^1\rangle.
\end{eqnarray*}

Since 
$$\langle f_n^0-f_{n,2}^0e_n^2,V^{*}e_n^1\rangle=\langle f_{n,1}^0e_n^1,V^{*}e_n^1\rangle=0$$
and 
$$|\langle f_n-f_n^0,V^{*}e_n^1\rangle|\leq\|f_n-f_n^0\|.\|V^{*}e_n^1\|\leq2\|P_n-P_n^0\|.\|\mathcal{P}\|,$$
we have
\begin{equation}
\label{bir}
|\langle Vf_n,e_n^1\rangle|\leq|p(n)|+2\|\mathcal{P}\|.\|P_n-P_n^0\|,
\end{equation}
and in a similar way, we can get
\begin{equation}
\label{iki}
|\langle Vf_n,e_n^2\rangle|\leq|q(-n)|+2\|\mathcal{Q}\|.\|P_n-P_n^0\|,
\end{equation}
where $p(n)$ and $q(n)$ are the Fourier coefficients of $\mathcal{P}$ and $\mathcal{Q}$. Observe that the sequences $|p(n)|$ and $|q(n)|$ are square summable since $p(n)$ and $q(n)$ are Fourier coefficients. In addition, $\|P_n-P_n^0\|$ is also square summable due to Theorem 7.1 in \cite{unconditional} since periodic and antiperiodic boundary conditions are regular. Hence, from (\ref{str}),(\ref{bir}) and (\ref{iki}) we get $\sum_{n\in\mathbb{Z}}|\lambda_n^{\mp}-n|^2<\infty.$ The proof of $\sum_{n\in\mathbb{Z}}|\lambda_n^--n|^2<\infty$ is also similar and $\sum_{n\in\mathbb{Z}}|\lambda_n^+-\lambda_n^-|^2<\infty$ comes from the triangle inequality.
\end{proof}
%It is important to note that $\|P_n-P_n^0\|$ is square summable for large enough $n$ by taking Theorem 15 in \cite{unconditional} into consideration. Hence, $|\langle Vf,e_n^1\rangle|$ and $|\langle Vf,e_n^2\rangle|$ are square summable by observing that $|p(n)|$ and $|q(-n)|$ are also in $\ell_2(\mathbb{Z})$. Therefore, the inequalities \ref{bir},\ref{iki} and \ref{str} lead to the fact that $|\lambda_n^+-n|$ is also square summable.

%Also, since $|\langle\psi_n-e_n^1,e_n^1\rangle|\leq\|P_n-P_n^0\|$
%$$\langle\psi_n,e_n^1\rangle=\langle e_n^1,e_n^1\rangle+\langle\psi_n-e_n^1,e_n^1\rangle\rightarrow 1.$$
%Now, we can write
%$$\lambda_n^+-n=\frac{\langle V\psi_n,e_n^1\rangle}{\langle\psi_n,e_n^1\rangle}$$
%t remains to prove $\langle V\psi_n,e_n^1\rangle\rightarrow 0.$
%We can rewrite this term in the following way,
%$$\langle V\psi_n,e_n^1\rangle=\langle Ve_n^1,e_n^1\rangle+\langle V(\psi_n-e_n^1),e_n^1\rangle$$
%and knowing that $\langle Ve_n^1,e_n^1\rangle=0$ we write
%$$|\langle V(\psi_n-e_n^1),e_n^1\rangle|=|\langle \psi_n-e_n^1,V^*e_n^1\rangle|\leq\|P_n-P_n^0\|.\|P\|$$ where 
%$$V^*=\begin{pmatrix}0 &\overline{Q} \\ \overline{P}&0\end{pmatrix}$$
%This completes the proof for the case $n$ is even and the proof is similar for the case $n$ is odd.

%\end{proof}

The next proposition gives estimates for $|\ell_i(f_n-f_n^0)|$ and $|\ell_i(\varphi_n-\varphi_n^0)|$, $i=0,1.$ The technique used in the proof of the proposition is based on a method developed by A. Batal (see Proposition 2.9 in \cite{ahmet} and Proposition 10 in \cite{ahmetarxive}).

\begin{proposition}\label{approx}
In the notation used above, there exists a sequence of positive real numbers $(\kappa_n)$ such that $\kappa_n\rightarrow 0$ and
\begin{equation}
\label{approx1}
|\ell_0(f_n-f_n^0)|\leq\kappa_n,
\end{equation}
\begin{equation}
\label{approx2}
|\ell_0(\varphi_n-\varphi_n^0)|\leq\kappa_n.
\end{equation}
\end{proposition}

\begin{proof}To obtain a clear notation, we fix and suppress the notation $n$ for the eigenvectors and put
$$f=f_n,\quad f^0=f_n^0,\quad\varphi^0=\varphi_n^0$$
and
$$f=\begin{pmatrix}
      f_1 \\
      f_2  
\end{pmatrix},\quad f^0=\begin{pmatrix}
      f_1^0 \\
      f_2^0  
\end{pmatrix},\quad \varphi=\begin{pmatrix}
      \varphi_1 \\
      \varphi_2  
\end{pmatrix},\quad \varphi^0=\begin{pmatrix}
      \varphi_1^0 \\
      \varphi_2^0  
\end{pmatrix}.$$

Now, it will be enough to find a sequence $(\kappa_n)$ converging to zero such that
\begin{equation}
\label{aprx1}
|f_i(0)-f_i^0(0)|\leq\kappa_n
\end{equation}
and
\begin{equation}
\label{aprx2}
|\varphi_i(0)-\varphi_i^0(0)|\leq\kappa_n
\end{equation}
for $i=1,2.$

We have $Lf=\lambda_n^+f$, and $i ({f^{0}_1})^{'}=n f_1^0.$
Subtracting these equations, we obtain
$$n(f_1-f_1^0)=i(f_1-f_1^0)^{'}+\mathcal{P}f_2-z_n^+f_1,$$
where $z_n^+=\lambda_n^+-n.$

Now, we assume that $n$ is even, then we have periodic eigenfunctions in the equation. We multiply both sides by $e^{i(n+1)x}$ and apply integration by parts on the first term of the right side. Then, we obtain
\begin{equation}\label{mmmm}
2i(f_1-f_1^0)(0)=I_1+I_2+I_3,
\end{equation}
where
$$I_1=-\int_0^{\pi}e^{i(n+1)x}(f_1-f_1^0)(x)dx,\quad I_2=\int_0^{\pi}e^{i(n+1)x}\mathcal{P}(x)f_2(x)dx$$
and
$$I_3=-\int_0^{\pi}z_n^+e^{i(n+1)x}f_1(x)dx.$$
For $I_1$, by Cauchy-Schwarz inequality and Lemma \ref{remark1} we have
$$|I_1|\leq\|f_1-f_1^0\|\leq\|f-f^0\|\leq\kappa_n.$$
To estimate $I_2$ recall that $f_2^0=C^0e^{inx}$ for some constant $C^0.$ Since $\|f^0\|=1$, we get $|C^0|\leq1$. Then, we have
\begin{eqnarray*}
|I_2|&\leq&\int_0^{\pi}|e^{i(n+1)x}\mathcal{P}(x)(P_n-P_n^0)(f_2)(x)|dx+|\int_0^{\pi}e^{i(n+1)x}\mathcal{P}(x)f_2^0(x)dx|\\
&\leq&\|\mathcal{P}\|.\|(P_n-P_n^0)\|+|C^0\int_0^{\pi}e^{i(n+1)x}e^{inx}\mathcal{P}(x)dx|
\end{eqnarray*}
The last term is a Fourier coefficient of the $L^2-$function $\mathcal{P}(x)$ which tends to zero as $n\rightarrow\infty.$

To obtain similar result for $I_3$, we immediately see that $|I_3|\leq|z_n^+|=|\lambda_n^+-n|$ and $z_n^+\rightarrow 0$ as $n\rightarrow\infty$ due to Lemma \ref{summable}. Hence, by (\ref{mmmm}) we get that  $|(f_1-f_1^0)(0)|$ tends to zero.

Estimation method for $(f_2-f_2^0)(0)$ can be continued by multiplying both sides of the equation 
$$n(f_2-f_2^0)=-i(f_2-f_2^0)^{'}+Qf_1-z_n^+f_2$$
by $e^{-i(n+1)x}$, and all the remaining argument is similar. So, this proves (\ref{approx1}).

Now, we recall that 
$$L\varphi=\lambda^+\varphi -\gamma_n\varphi+\xi_n f,  \quad L^0\varphi^0=n\varphi^0.$$
We subtract the second equation from the first equation and write the equation of the first components
\begin{equation}
\label{varphi}
i(\varphi_1-\varphi_1^0)^{'}+\mathcal{P}\varphi_2=n(\varphi_1-\varphi_1^0)+z_n^+\varphi_1-\gamma_n\varphi_1+\xi_nf_1.
\end{equation}
After multiplying both sides of (\ref{varphi}) by $e^{i(n+1)x}$, we integrate and obtain
$$-2i(\varphi_1-\varphi_1^0)(0)=J_1+J_2+J_3+J_4+J_5,$$
where
$$J_1=-\int_0^{\pi}e^{i(n+1)x}(\varphi_1-\varphi_1^0)(x)dx,\quad J_2=-\int_0^{\pi}e^{i(n+1)x}\mathcal{P}(x)\varphi_2(x)dx,$$
$$J_3=z_n^+\int_0^{\pi}e^{i(n+1)x}\varphi_1(x)dx,\quad J_4=-\gamma_n\int_0^{\pi}e^{i(n+1)x}\varphi_1(x)dx$$
and
$$J_5=\xi_n\int_0^{\pi}e^{i(n+1)x}f_1(x)dx.$$
The estimations for $J_1,J_2$ and $J_3$ are very similar to those for $I_1,I_2$ and $I_3$ respectively.

Lemma 40 together with Proposition 35 in \cite{instability} gives that $\gamma_n\rightarrow 0$, additionally Lemma 59, Lemma 60 and Proposition 35 in \cite{instability} imply that $\xi_n\rightarrow 0$ as $n$ goes to infinity. So, $J_4$ and $J_5$ are also dominated by a sequence converging to zero. Hence, $|(\varphi_1-\varphi_1^0)(0)|$ tends to zero.

To estimate $|(\varphi_2-\varphi_2^0)(0)|$, we follow similar calculations using $e^{-i(n+1)x}$ instead of $e^{i(n+1)x}.$ Furthermore, the way we prove the result for the case when $n$ is odd is also similar.
\end{proof}

\begin{Remark}\label{remarkx}In view of (\ref{aprx1}) and (\ref{aprx2})
\begin{equation*}
|f_n^i(0)-f_{n,i}^0|<\kappa_n\quad\text{ and }\quad |\varphi_n^i(0)-\varphi_{n,i}^0|<\kappa_n
\end{equation*}
for $i=1,2$, where $\kappa_n\rightarrow 0$
\end{Remark}

There are functionals $\alpha_n(V;z)$ and $\beta_n^{\pm}(V;z)$ defined for large enough $|n|$, $n\in \mathbb{Z}$ and $|z|<1/2$ such that $\lambda=n+z$ is (periodic if $n$ is even or  antiperiodic if $n$ is odd) eigenvalue of $L$ if and only if $z$ is an eigenvalue of the matrix $S(\lambda)$

$$\begin{pmatrix}\alpha_n(V;z)&\beta_n^-(V;z)\\ \beta_n^+(V;z)&\alpha_n(V;z)\end{pmatrix}.$$
by Lemma 21 in \cite{instability}. Furthermore, $z_n^{\pm}=\lambda_n^{\pm}-n$ are the only solutions of the basic equation

$$(z-\alpha_n(V;z))^2=\beta_n^-(V;z)\beta_n^+(V;z),$$ where $|z|<1/2$. If $Lf_n=\lambda_n^+ f_n$ then $f_n^0$ is an eigenvector of the operator $L^0+S(\lambda_n^+):E_n^0\rightarrow E_n^0$ with a corresponding eigenvalue $\lambda_n^+$, and one may write the following system
\begin{equation}\label{linearequation}
\begin{pmatrix}z_n^+-\alpha_n(z_n^+)&-\beta_n^-(z_n^+)\\-\beta_n^+(z_n^+)&z_n^+-\alpha_n(z_n^+)\end{pmatrix}
\begin{pmatrix}f_{n,1}^0\\f_{n,2}^0\end{pmatrix}=\begin{pmatrix}0\\0\end{pmatrix}.
\end{equation}
So, we get

\begin{equation}\label{maineq}
(z_n^+-\alpha_n(z_n^+))^2=\beta_n^-(z_n^+).\beta_n^+(z_n^+).
\end{equation}

%$$f^0=\begin{pmatrix}f_1^0\\f_2^0\end{pmatrix}=\frac{1}{\sqrt{\rho(z^+)}}
%\begin{pmatrix}\sqrt{\beta^-(z^+)}\\\sqrt{\beta^+(z^+)}\end{pmatrix}$$
We put $\varphi_n^0=\varphi_{n,1}^0e_n^1+\varphi_{n,2}^0e_n^2.$ Let $\varphi_n^0=c_1f_n^0+c_2(f_n^0)^{\bot},$ where 
$$(f_n^0)^{\bot}=\overline{f_{n,2}^0}e_n^1-\overline{f_{n,1}^0}e_n^2.$$
Then,

$$c_1=\langle\varphi_n^0,f_n^0\rangle=\langle\varphi_n^0-\varphi_n,f_n\rangle+\langle\varphi_n^0,f_n^0-f_n\rangle=O(\kappa_n),$$
and

$$|c_2|=\sqrt{1-|c_1|^2}=1+O(\kappa_n).$$

Hence, without loss of generality we may write

\begin{equation}\label{test1}
\varphi_{n,1}^0=\overline{f_{n,2}^0}+O(\kappa_n)\ , \quad \varphi_{n,2}^0=-\overline{f_{n,1}^0}+O(\kappa_n)\ .
\end{equation}

We have $\tilde{g}_n^0=e^{i\theta}g_n^0$ for some $\theta$ because the eigenspace of the free operator under general boundary conditions is one dimensional as stated in Theorem \ref{localizationgeneral}. Without loss of generality, we can put $g_n^0=\tilde{g}_n^0.$

The following two equations are due to Proposition \ref{approx}:

\begin{equation}\label{app1}
\ell_0(f_n)=\ell_0(f_n^0)+O(\kappa_n)\ ,\quad \ell_0(\varphi_n)=\ell_0(\varphi_n^0)+O(\kappa_n)
\end{equation}

By Lemma \ref{remark1}, we also obtain other estimations as
\begin{align*}|\langle f_n,\tilde{g}_n\rangle-\langle f_n^0,g_n^0\rangle|&\leq|\langle f_n-f_n^0,\tilde{g}_n\rangle|+|\langle f_n^0,\tilde{g}_n-g_n^0\rangle|\\&\leq \|f_n-f_n^0\|.\|\tilde{g}_n\|+\|f_n^0\|.\|\tilde{g}_n-g_n^0\|\leq2\kappa_n,
\end{align*}
where $\kappa_n\rightarrow 0.$ Similarly, we get

$$|\langle \varphi_n,\tilde{g}_n\rangle-\langle
\varphi_n^0,g_n^0\rangle|\leq2\kappa_n.$$
Hence, one may write

\begin{equation*}
\langle f_n,\tilde{g}_n\rangle=\langle f_n^0,g_n^0\rangle+O(\kappa_n),\quad
\langle \varphi_n,\tilde{g}_n\rangle=\langle
\varphi_n^0,g_n^0\rangle+O(\kappa_n).
\end{equation*}

%%From now on, we assume $n$ to be even and we put $D=\frac{-a}{1+b}$
%%and can write,

Now, let us write $g_n^0$ as

$$g_n^0=A_ne_n^1+B_ne_n^2.$$

Together with (\ref{app1}), we have

$$\ell_0(\varphi_n)=(1+(-1)^n b)\varphi_{n,1}^0+a\varphi_{n,2}^0+O(\kappa_n).$$
%and

%$$\ell_0(f_n)=(1+(-1)^nb)f_{n,1}^0+af_{n,2}^0+O(\kappa_n).$$

%Similarly, we also may write

%$$\ell_1(\varphi)=(-1)^n d\varphi_1^0+(c+(-1)^n 1)\varphi_2^0+O(\kappa)$$
%and
%$$\ell_1(f)=(-1)^n df_1^0+(c+(-1)^n 1)f_2^0+O(\kappa)$$

%% so we here express the main estimations...in terms of f_1^0 and f_2^0..
So, we conclude that

\begin{equation}\label{est1}
\ell_0(\varphi_n)=(1+(-1)^n b)\overline{f_{n,2}^0}-a\overline{f_{n,1}^0}+O(\kappa_n)
\end{equation}

\begin{equation}\label{est2}
\ell_0(f_n)=(1+(-1)^nb)f_{n,1}^0+af_{n,2}^0+O(\kappa_n)
\end{equation}

%\begin{equation}\label{est3}
%\ell_1(\varphi)=(-1)^n d\overline{f_2^0}-(c+(-1)^n)\overline{f_1^0}+O(\kappa)
%\end{equation}

%\begin{equation}\label{est4}
%\ell_1(f)=(-1)^n df_1^0+(c+(-1)^n1)f_2^0+O(\kappa)
%\end{equation}

\begin{equation}\label{ilker4}
\langle
\varphi_n,\tilde{g}_n\rangle=\overline{A_nf_{n,2}^0}-\overline{B_nf_{n,1}^0}+O(\kappa_n)
\end{equation}

\begin{equation}\label{ilker5}
\langle
f_n,\tilde{g}_n\rangle=\overline{A_n}f_{n,1}^0+\overline{B_n}f_{n,2}^0+O(\kappa_n).
\end{equation}

We now get a nonzero approximation for $\langle G_n,\tilde{g}\rangle.$

\begin{Lemma}\label{nonzero}
$\tau_n^{-1}\langle G_n,\tilde{g}_n\rangle=C+O(\kappa_n)$ for some constant $C\neq 0$, where $C$ depends on the general boundary conditions given by (\ref{ilker2})  and (\ref{ilker3}) with the restriction $ad\neq0.$
\end{Lemma}
\begin{proof}
Recall (\ref{linear1}) as
$$\tau_n^{-1}\langle G_n,\tilde{g}_n\rangle=\ell_0(\varphi_n)\langle
f_n,\tilde{g}_n\rangle-\ell_0(f_n)\langle\varphi_n,\tilde{g}_n\rangle.$$

We substitute all the estimations found by (\ref{est1}), (\ref{est2}), (\ref{ilker4}) and (\ref{ilker5}) into the equation
above, and after an easy calculation, obtain

\begin{align*}
  \tau_n^{-1}\langle G_n,\tilde{g}_n\rangle &=[(1+(-1)^n b)\overline{f_{n,2}^0}-a\overline{f_{n,1}^0}+O(\kappa_n)].[\overline{A_n}f_{n,1}^0+\overline{B_n}f_{n,2}^0+O(\kappa_n)] \\
                             &\ -[(1+(-1)^nb)f_{n,1}^0+af_{n,2}^0+O(\kappa_n)].[\overline{A_nf_{n,2}^0}-\overline{B_nf_{n,1}^0}+O(\kappa_n)]
                             \\
                             &=[(1+(-1)^n b)\overline{B_n}-a\overline{A_n}]|f_{n,1}^0|^2+[(1+(-1)^n b)\overline{B_n}-a\overline{A_n}]|f_{n,2}^0|^2+O(\kappa_n)
                             \\
                             &=(1+(-1)^n b)\overline{B_n}-a\overline{A_n}+O(\kappa_n).
\end{align*}

By (\ref{onedimension1}) we have

\begin{align*}
\tau_n^{-1}\langle G_n,\tilde{g}_n\rangle&=(1+(-1)^n b)\overline{B_n}-a\overline{A_n}+O(\kappa_n)\\
&=(1+(-1)^n b)\frac{\overline{-\overline{a}}}{\sqrt{|a|^2+|1-(-1)^n b|^2}}-a\frac{\overline{1-(-1)^n \overline{b}}}{\sqrt{|a|^2+|1-(-1)^n b|^2}}+O(\kappa_n)\\
&=\frac{-2a}{\sqrt{|a|^2+|1-(-1)^n b|^2}}+O(\kappa_n).
\end{align*}
So, the result follows because $a\neq0$.
%Now, we assume $b=\pm1$, then \ref{est3} and \ref{est4}
%$$\tau(G)^{-1}\langle G,\tilde{g}\rangle=\ell_1(\varphi)\langle
%f,\tilde{g}\rangle-\ell_1(f)\langle\varphi,\tilde{g}\rangle$$
%which, similarly, \ref{onedimension1}, \ref{onedimension2}, \ref{est3} and \ref{est4} lead to
%\begin{align*}
%\tau(G)^{-1}\langle G,\tilde{g}\rangle&=\mp d\overline{B}-(c\mp1)\overline{A}+O(\kappa)\\
%&=\mp d\frac{-\overline{d}}{\sqrt{4+|d|^2}}-(c\mp1)\frac{2}{\sqrt{4+|d|^2}}+O(\kappa)\\
%&=\frac{\pm|d|^2-(c\mp1).2}{\sqrt{4+|d|^2}}+O(\kappa).
%\end{align*}
%Note that if $b=1$, then $c=-1$, and if $b=-1$, then $c=1$. Hence, the proof is complete.
\end{proof}

\begin{proposition}\label{prop1}
There are constants $D_1,D_2>0$ such that for $n\in\mathbb{Z}$ with large enough $|n|$,
$$|\mu_n-\lambda_n^+|\leq D_1|\gamma_n|+D_2(|B_n^+|+|B_n^-|),$$
where $B_n^{\pm}=\beta_n^{\pm}(z_n^+)$.
\end{proposition}
\begin{proof}
We consider the equation (\ref{ilker6}). We assume that
$0<\kappa_n\leq\frac{|C|}{2}$. We multiply both sides of the equation (\ref{ilker6}) by $\tau_n^{-1}$ and get
\begin{equation}\label{pppp}
\tau_n^{-1}(\mu_n-\lambda_n^+)\langle G_n,\tilde{g}_n\rangle=\tau_n^{-1}t_n(\xi_n \langle
f_n,\tilde{g}_n\rangle-\gamma_n\langle\varphi_n,\tilde{g}_n\rangle).
\end{equation}
Lemma \ref{nonzero} guarantees that we may divide both sides of (\ref{pppp}) by $\tau_n^{-1}\langle
G_n,\tilde{g}_n\rangle$. We also have the inequality
\begin{equation}\label{sesese}
|\xi_n|\leq4|\gamma_n|+2(|B_n^+|+|B_n^-|)
\end{equation}
due to Lemma 59 and Lemma 60 in \cite{instability}. 

Note that since $|A_n|^2+|B_n|^2=1$ and $|f_{n,1}^0|^2+|f_{n,2}^0|^2=1$, by Cauchy-Schwarz inequality $|\overline{A_nf_{n,2}^0}-\overline{B_nf_{n,1}^0}|\leq1$ and $|\overline{A_n}f_{n,1}^0+\overline{B_n}f_{n,2}^0|\leq1$. Then, we obtain the following inequality by using the estimations (\ref{est1}), (\ref{ilker4}), (\ref{ilker5}) and (\ref{sesese}):
\begin{align*}|\mu_n-\lambda_n^+|&=\frac{|\tau_n^{-1}||t_n(\xi_n \langle
f_n,\tilde{g}_n\rangle-\gamma_n\langle\varphi_n,\tilde{g}_n\rangle)|}{|\tau_n^{-1}||\langle
G_n,\tilde{g}_n\rangle|}\\
&=\frac{|\ell_0(f_n)||(\xi_n \langle
f_n,\tilde{g}_n\rangle-\gamma_n\langle\varphi_n,\tilde{g}_n\rangle)|}{|\tau_n^{-1}||\langle
G_n,\tilde{g}_n\rangle|}\\
&\leq\frac{(|\ell_0(f_n^0)|+\kappa_n)[|\xi_n|(|\overline{A_n}f_{n,1}^0+\overline{B_n}f_{n,2}^0|+\kappa_n)+|\gamma_n|(|\overline{A_nf_{n,2}^0}-\overline{B_nf_{n,1}^0}|+\kappa_n)]}{|C|-\kappa_n}\\
&\leq\bigg(\frac{|(1+(-1)^n b)f_{n,1}^0+af_{n,2}^0|+\frac{|C|}{2}}{\frac{|C|}{2}}\bigg)(1+|C|)|\xi_n|\\
&+\bigg(\frac{|(1+(-1)^n b)f_{n,1}^0+af_{n,2}^0|+\frac{|C|}{2}}{\frac{|C|}{2}}\bigg)(1+|C|)|\gamma_n|.
\end{align*}
So, we get the result as
\begin{align*}
|\mu_n-\lambda_n^+|&\leq\bigg(\frac{2+2|b|+2|a|+|C|}{|C|}\bigg)(1+|C|)|\xi_n|+\bigg(\frac{2|1+b|+2|a|+|C|}{|C|}\bigg)(1+|C|)|\gamma_n|\\
&\leq\bigg(\frac{2+2|b|+2|a|+|C|}{|C|}\bigg)(1+|C|)(4|\gamma_n|+2(|B_n^+|+|B_n^-|))\\
&+\bigg(\frac{2+2|b|+2|a|+|C|}{|C|}\bigg)(1+|C|)|\gamma_n|\\
&=D_1|\gamma_n|+D_2(|B_n^+|+|B_n^-|),\\
\end{align*}
where
$$D_1=5\cdot\bigg(\frac{2+2|b|+2|a|+|C|}{|C|}\bigg)(1+|C|) \quad \text{and}\quad D_2=2\cdot\bigg(\frac{2+2|b|+2|a|+|C|}{|C|}\bigg)(1+|C|).$$
%Apparently, $D_1$ and $D_2$ depend on the general boundary conditions.
\end{proof}

\section{Estimation for $|\mu_n-\lambda_n^+|+|\lambda_n^+-\lambda_n^-|$}
We start with a generalized version of Proposition 63 in \cite{instability}.
\begin{proposition}\label{prop2}
Let \ $M>1$ be a fixed number, then for $n\in Z$ with sufficiently large $|n|$ if

\begin{equation}\label{assumption}
\frac{1}{M}|B_n^-|\leq|B_n^+|\leq M|B_n^-|
\end{equation}
then
$$|\beta_n^-(z_n^*)|+|\beta_n^+(z_n^*)|\leq\frac{1+M}{\sqrt{M}}|\gamma_n|,$$ where $B_n^{\pm}=\beta_n^{\pm}(z_n^+)$ and
$z_n^*=(\lambda_n^++\lambda_n^-)/2-n$ in the case of
simple eigenvalues and $z_n^*=\lambda_n^+-n$ otherwise.
\end{proposition}

\begin{proof}
We mainly follow the proof of Propositon 63 in \cite{instability}.

The case $B_n^+=B_n^-=0$ is explained in the proof of Proposition 63 in \cite{instability}.
Assume $B_n^+B_n^-\neq 0$ and $\gamma_n\neq 0$. Since
$t_n=\frac{|B_n^+|}{|B_n^-|}\in [\frac{1}{M},M]$, we have 
$$\sqrt{M}-\sqrt{t_n}\geq\frac{\sqrt{M}-\sqrt{t_n}}{\sqrt{t_n}.\sqrt{M}}$$
because $1\geq\frac{1}{\sqrt{t_nM}}\cdot$ Then, we get
$$\frac{1}{\sqrt{M}}+\sqrt{M}\geq\frac{1}{\sqrt{t_n}}+\sqrt{t_n}$$
which leads to
$$\frac{2\sqrt{t_n}}{1+t_n}\geq\frac{2\sqrt{M}}{1+M}\cdot$$
In Lemma 49 in \cite{instability} we take
$$\delta_n<\frac{\sqrt{M}}{1+M}$$ for sufficiently large $|n|$, then
\begin{align*}|\gamma_n|&\geq(\frac{2\sqrt{t_n}}{1+t_n}-\delta_n)(|\beta_n^-(z_n^*)|+|\beta_n^+(z_n^*)|)\geq(\frac{2\sqrt{M}}{1+M}-\frac{\sqrt{M}}{1+M})(|\beta_n^-(z_n^*)|+|\beta_n^+(z_n^*)|),\end{align*}
which gives us the result.

Now, if $B_n^+B_n^-\neq 0$ and $\gamma_n=0,$ then $z_n^+$ is the only root of the equation (\ref{maineq}) or zero of the function
$$h_n(z)=(\zeta_n(z))^2-\beta_n^+(z)\beta_n^-(z)$$
in the disc $D=\{z:|z|<1/8\}$, where $\zeta_n(z)=z-\alpha_n(z)$. The functions $h_n$ is analytic on $D$ since $\beta_n^{\mp}(z)$ and $\alpha_n(z)$ are analytic funtions on $D$ (see Proposition 28 in \cite{instability}). We also have
$$|h_n(z)-z^2|=|z^2+\alpha_n(z)^2-2z\alpha_n(z)-\beta_n^-(z)\beta_n^+(z)-z^2|=|\alpha_n(z)^2-2z\alpha_n(z)-\beta_n^-(z)\beta_n^+(z)|.$$
By Proposition 35 in \cite{instability}, the maximum values of $|\alpha_n(z)|$ and $|\beta_n^{\mp}(z)|$ on the boundary of $D$ converge to zero as $|n|\rightarrow\infty$, hence we may write for sufficiently large $|n|$
$$\sup_{\partial D}|h_n(z)-z^2|<\sup_{\partial D}|z^2|.$$
By Rouch\'{e}'s theorem, $z^+$ is a double root of the equation $h_n(z)=0$ which leads to $h_n'(z^+)=0,$ so the following holds
$$2\zeta_n(z^+)\cdot(1-\frac{d \alpha_n}{dz}(z^+))=\frac{d\beta_n^+}{dz}(z^+)\cdot\beta_n^-(z^+)+\beta_n^+(z^+)\cdot\frac{d\beta_n^-}{dz}(z^+).$$

If we consider the upper bounds $$|{\frac{d\alpha_n}{dz}(z^+)}|\leq\frac{1}{\sqrt{M}(M+1)+1},\quad
|{\frac{d\beta_n^{\pm}}{dz}(z^+)}|\leq\frac{1}{\sqrt{M}(M+1)+1}$$ for sufficiently large $|n|$ by Proposition 35 in \cite{instability}, then triangle inequality gives
$$2|\zeta_n(z^+)|(1-\frac{1}{\sqrt{M}(M+1)+1})\leq\frac{1}{\sqrt{M}(M+1)+1}(|B_n^+|+|B_n^-|),$$
where $|\zeta_n(z^+)|=\sqrt{|B_n^+B_n^-|}$ by the basic equation (\ref{maineq}). Hence, we have 
$$2(\sqrt{M}(M+1))\leq\frac{|B_n^+|+|B_n^-|}{\sqrt{|B_n^+B_n^-|}}\leq\frac{\sqrt{M}(M+1)|B_n^-|}{\sqrt{|B_n^-B_n^-|}}$$
by (\ref{assumption}) and the above inequality, which gives a contradiction $2\leq1$. Hence, the proof is complete.
\end{proof}
Now, we consider the complementary cases

\begin{equation}\label{subcases}
(i)\ M|B_n^+|<|B_n^-|,\ \ \ (ii)\ M|B_n^-|<|B_n^+|.
\end{equation}

\begin{Lemma}\label{lem1}
If $(i)$ in (\ref{subcases}) is true  and $|n|$ is sufficiently large, then the
followings hold $$|f_{n,2}^0|\leq\frac{1}{\sqrt{M+1}},\quad |f_{n,1}^0|\geq\frac{\sqrt{M}}{\sqrt{M+1}},$$
%$$|f_{n,1}^0|\geq\frac{\sqrt{M}}{\sqrt{M+1}},$$
$$|\varphi_{n,2}^0|>\frac{\sqrt{M}}{\sqrt{M+1}}(1-\kappa_n),\quad |\varphi_{n,1}^0|<\frac{1}{\sqrt{M+1}}+4\sqrt{\kappa_n},$$
%$$|\varphi_{n,1}^0|<\frac{1}{\sqrt{M+1}}+4\sqrt{\kappa_n},$$
where $(\kappa_n)$ is a sequence of positive numbers such that $\kappa_n\rightarrow0.$
\end{Lemma}

\begin{proof}
We follow the proof of Lemma 64 in \cite{instability}.

Let $M>1$ and suppose $M|B_n^+|<|B_n^-|.$
By (\ref{linearequation}), we have the followings
$$\xi_n^+f_{n,1}^0+B_n^-f_{n,2}^0=0,\quad (\xi_n^+)^2=B_n^-B_n^+,$$
which gives
$$|B_n^-||f_{n,2}^0|=|\xi^+||f_{n,1}^0|=\sqrt{|B_n^-B_n^+|}|f_{n,1}^0|.$$
Then, we obtain
$$|f_{n,2}^0|=\frac{\sqrt{|B_n^+|}}{\sqrt{|B_n^-|}}.|f_{n,1}^0|\leq\frac{1}{\sqrt{M}}|f_{n,1}^0|$$
and
$$(M+1)|f_{n,2}^0|^2\leq|f_{n,1}^0|^2+|f_{n,2}^0|^2=1,$$ so
\begin{equation}\label{ub1}
|f_{n,2}^0|\leq\frac{1}{\sqrt{M+1}}.
\end{equation}
And also we get a lower bound for $|f_{n,1}^0|$
$$1=|f_{n,1}^0|^2+|f_{n,2}^0|^2\leq|f_{n,1}^0|^2+\frac{|f_{n,1}^0|^2}{M}=\frac{M+1}{M}|f_{n,1}^0|^2.$$
Hence
\begin{equation}\label{lb1}
|f_{n,1}^0|\geq\frac{\sqrt{M}}{\sqrt{M+1}}\cdot
\end{equation}
Now, we need to find bounds for $|\varphi_{n,1}^0|$ and $|\varphi_{n,2}^0|$. Noting that $f_n \bot\varphi_n$, we obtain
$$|\langle
f_n^0,\varphi_n^0\rangle|\leq|\langle
f_n^0-f_n,\varphi_n^0\rangle|+|\langle
f_n,\varphi_n^0-\varphi_n\rangle|\leq\|f_n-f_n^0\|\|\varphi_n^0\|+\|f_n\|\|\varphi_n-\varphi_n^0\|\leq2\kappa_n$$
since $\|f_n-f_n^0\|\leq\kappa_n$ and $\|\varphi_n-\varphi_n^0\|\leq\kappa_n$ by Lemma \ref{remark1}. On the other hand,
\begin{align*}
|\langle
f_n^0,\varphi_n^0\rangle|&=|f_{n,1}^0\overline{\varphi_{n,1}^0}+f_{n,2}^0\overline{\varphi_{n,2}^0}|\leq2\kappa_n,
\end{align*}
which leads to
$$|f_{n,1}^0\varphi_{n,1}^0|\leq|f_{n,2}^0\varphi_{n,2}^0|+2\kappa_n.$$
Then, by (\ref{ub1}) and (\ref{lb1}),
$$|\varphi_{n,1}^0|\leq\frac{|f_{n,2}^0|}{|f_{n,1}^0|}|\varphi_{n,2}^0|+\frac{2\kappa_n}{|f_{n,1}^0|}\leq\frac{|\varphi_{n,2}^0|}{\sqrt{M}}+\frac{2\kappa_n\sqrt{M+1}}{\sqrt{M}}$$
and
\begin{align*}
1&=|\varphi_{n,1}^0|^2+|\varphi_{n,2}^0|^2\leq\frac{|\varphi_{n,2}^0|^2}{M}+\frac{4\kappa_n^2(M+1)}{M}+\frac{4|\varphi_{n,2}^0|\kappa_n\sqrt{M+1}}{M}+|\varphi_{n,2}^0|^2\\
&\leq\frac{|\varphi_{n,2}^0|^2}{M}+4\kappa_n\frac{\kappa_n(M+1)+\sqrt{M+1}}{M}+|\varphi_{n,2}^0|^2\\
&\leq\frac{|\varphi_{n,2}^0|^2}{M}+4\kappa_n\frac{\kappa_n(M+M)+M+M}{M}+|\varphi_{n,2}^0|^2<\frac{|\varphi_{n,2}^0|^2}{M}+16\kappa_n+|\varphi_{n,2}^0|^2.
\end{align*}
Now, we have
\begin{equation}\label{lb2}
|\varphi_{n,2}^0|>\frac{\sqrt{M}}{\sqrt{M+1}}\sqrt{1-16\kappa_n}.
\end{equation}
On the other hand,
\begin{align*}
1&=|\varphi_{n,1}^0|^2+|\varphi_{n,2}^0|^2\geq|\varphi_{n,1}^0|^2+\frac{M}{M+1}(1-16\kappa_n),
\end{align*}
which leads to
$$\frac{1+16M\kappa_n^2}{M+1}\geq |\varphi_{n,1}^0|^2.$$
Since we have $\sqrt{x+y}\leq\sqrt{x}+\sqrt{y}$ for any $x,y\geq0$, we may write as follows
$$\frac{1}{\sqrt{M+1}}+\frac{4\sqrt{M}\sqrt{\kappa_n}}{\sqrt{M+1}}\geq|\varphi_{n,1}^0|,$$
which gives
\begin{equation}\label{ub2}
|\varphi_{n,1}^0|<\frac{1}{\sqrt{M+1}}+4\sqrt{\kappa_n}\ .
\end{equation}

\end{proof}
An analogy of the previous lemma can be given for the case $(ii)$ in (\ref{subcases}) which has a very similar proof.

\begin{Lemma}\label{lem2}
If $(ii)$ in (\ref{subcases}) is true and $|n|$ is sufficiently large, then the
followings hold 
$$|f_{n,1}^0|\leq\frac{1}{\sqrt{M+1}},\quad |f_{n,2}^0|\geq\frac{\sqrt{M}}{\sqrt{M+1}},$$
%$$|f_{n,2}^0|\geq\frac{\sqrt{M}}{\sqrt{M+1}},$$
$$|\varphi_{n,1}^0|>\frac{\sqrt{M}}{\sqrt{M+1}}(1-\kappa_n),\quad |\varphi_{n,2}^0|<\frac{1}{\sqrt{M+1}}+4\sqrt{\kappa_n},$$
%$$|\varphi_{n,2}^0|<\frac{1}{\sqrt{M+1}}+4\sqrt{\kappa_n},$$
where $(\kappa_n)$ is a sequence of positive numbers such that $\kappa_n\rightarrow0.$
\end{Lemma}

The next lemma gives an estimation for the ratio $\frac{|\ell_0(f_n)|}{|\ell_0(\varphi_n)|}$ in the two cases $(i)$ and $(ii)$ in (\ref{subcases}).

\begin{Lemma}\label{zordu1}
Suppose $b\neq\pm1$. If $|n|$ is sufficiently large and one of the cases $(i)$ and $(ii)$ in (\ref{subcases}) is
true for 
\begin{equation}\label{lastlast}
M>max\{4(|a|/|1+b|)^2,4(|a|/|1-b|)^2,4(|1+b|/|a|)^2,4(|1-b|/|a|)^2\},
\end{equation}
then there are
constants $D_3>0$ and $D_4>0$ such that
\begin{equation}\label{result}
D_3<\frac{|\ell_0(f_n)|}{|\ell_0(\varphi_n)|}<D_4
\end{equation}
holds. 
\end{Lemma}
\begin{proof}
We mainly follow the proof of Lemma 64 in \cite{instability}.

Suppose the case $(i)$ in (\ref{subcases}) is true. Now, in order to get the inequality (\ref{result}) we have to find lower bounds and upper bounds for both $|\ell_0(f_n)|$ and $|\ell_0(\varphi_n)|$. We easily get upper bounds as
\begin{equation}\label{g1g1}
|\ell_0(f_n)|\leq|\ell_0(f_n^0)|+\kappa_n\leq1+|b|+|a|+\kappa_n
\end{equation}
and
\begin{equation}\label{g2g2}
|\ell_0(\varphi_n)|\leq|\ell_0(\varphi_n^0)|+\kappa_n\leq1+|b|+|a|+\kappa_n
\end{equation}
by (\ref{app1}). 

We continue to get lower bounds for $|\ell_0(f_n)|$ and $|\ell_0(\varphi_n)|$ by using the results coming from Lemma \ref{lem1} as follows

\begin{align*}
|\ell_0(f_n)|&\geq|\ell_0(f_n^0)|-\kappa_n=|(1+(-1)^nb)f_{n,1}^0+af_{n,2}^0|-\kappa_n\\
%&\geq||(1+(-1)^nb)f_{n,1}^0|-|af_{n,2}^0||-\kappa_n\\
&\geq|(1+(-1)^nb)||f_{n,1}^0|-|a||f_{n,2}^0|-\kappa_n\geq|(1+(-1)^nb)|\frac{\sqrt{M}}{\sqrt{M+1}}-|a|\frac{1}{\sqrt{M+1}}-\kappa_n\\
&=\frac{|(1+(-1)^nb)|\sqrt{M}-|a|}{\sqrt{M+1}}-\kappa_n.
\end{align*}
Then, since $M>4(|a|/|1\mp b|)^2$ by (\ref{lastlast}) we get
$$|\ell_0(f_n)|\geq\frac{|a|}{\sqrt{M+1}}-\kappa_n.$$

We also obtain the inequality

\begin{align*}
|\ell_0(\varphi_n)|&\geq|\ell_0(\varphi_n^0)|-\kappa_n=|(1+(-1)^nb)\varphi_{n,1}^0+a\varphi_{n,2}^0|-\kappa_n\\
%&\geq||(1+(-1)^nb)\varphi_{n,1}^0|-|a\varphi_{n,2}^0||-\kappa_n\\
&\geq|a|.|\varphi_{n,2}^0|-|(1+(-1)^nb)|.|\varphi_{n,1}^0|-\kappa_n\\&\geq|a|\frac{\sqrt{M}}{\sqrt{M+1}}(1-\kappa_n)-|(1+(-1)^nb|(\frac{1}{\sqrt{M+1}}+4\sqrt{\kappa_n})-\kappa_n.
\end{align*}
Since $M>4(|1\mp b|/|a|)^2$ by (\ref{lastlast}) we get
\begin{align*}
|\ell_0(\varphi_n)|&\geq\frac{|1+(-1)^nb|}{\sqrt{M+1}}-(|a|\frac{\sqrt{M}}{\sqrt{M+1}}+1)\kappa_n+4|1+(-1)^nb|\sqrt{\kappa_n}.
\end{align*}

In case when $n$ is even,
%$$\kappa_n<\frac{|1+b|.|a|}{2(|1+b|+|a|+1)\sqrt{4|1+b|^2+|a|^2}},$$
we have
$$|\ell_0(f_n)|\geq\frac{|a|}{2\sqrt{M+1}}, \quad |\ell_0(\varphi_n)|\geq\frac{|1+b|}{2\sqrt{M+1}}$$
for sufficiently large $|n|.$
Hence, by (\ref{g1g1}) and (\ref{g2g2}) for $\kappa_n<1$, we may conclude that
$$\frac{\frac{|a|}{2\sqrt{M+1}}}{1+|b|+|a|}\leq\frac{\frac{|a|}{2\sqrt{M+1}}}{1+|b|+|a|+\kappa_n}\leq\frac{|\ell_0(f_n)|}{|\ell_0(\varphi_n)|}\leq\frac{1+|b|+|a|+\kappa_n}{\frac{|1+b|}{2\sqrt{M+1}}}\leq\frac{2+|b|+|a|}{\frac{|1+b|}{2\sqrt{M+1}}}$$
which leads to (\ref{result}) with 
$$D_3=\frac{\frac{|a|}{2\sqrt{M+1}}}{1+|b|+|a|}\ , \quad\quad D_4=\frac{2+|b|+|a|}{\frac{|1+b|}{2\sqrt{M+1}}}\cdot$$

Similar result holds in the case when $n$ is odd. Also, the proof for the case $(ii)$ in (\ref{subcases}) is the same as the proof for the case $(i)$ in (\ref{subcases}).

\end{proof}

We give an analogue of Proposition 65 in \cite{instability} with its similar proof.
\begin{proposition}\label{prop3}
If $(i)$ or $(ii)$ in (\ref{subcases}) holds for 
\begin{equation}\label{moremore}
M>max\{4(|1- b|/|a|)^2,4(|1+ b|/|a|)^2,4(|a|/|1-b|)^2,4(|a|/|1+b|)^2\},
\end{equation}
where $b\neq\mp1$ (or equivalently $ad\neq0$), then
there are constants $D_8>0$ and $D_9>0$ such that
\begin{equation}
|B_n^+|+|B_n^-|\leq D_8|\gamma_n|+D_9|\mu_n-\lambda_n^+|
\end{equation}
holds.
\end{proposition}
\begin{proof}
We prove the cases when $n$ is even and $n$ is odd simultanously. Assume the case $(i)$ in (\ref{subcases}) for the given $M$ in the hypothesis.
We know that $\ell_0(G_n)=0$ where $G_n=s_nf_n+t_n\varphi_n$, so
$|\ell_0(f_n)|/|\ell_0(\varphi_n)|=|t_n|/|s_n|$ and by Lemma \ref{zordu1}
$$0<D_3\leq|t_n|/|s_n|\leq D_4$$ and
$$1=|s_n|^2+|t_n|^2\leq\frac{|t_n|^2}{D_3^2}+|t_n|^2,$$ so we obtain
\begin{equation}\label{lower1}
|t_n|\geq\sqrt{\frac{D_3^2}{1+D_3^2}}>0.
\end{equation}
Furthermore, the inequality $M|B_n^+|<|B_n^-|$ and its consequence with Lemma \ref{lem1} give the following estimation to get a lower bound for $|\langle f_n,\tilde{g}_n\rangle|$ by making use of (\ref{ilker5}), (\ref{moremore}) and the coefficients in (\ref{onedimension1}) as
\begin{align*}
|\langle f_n,\tilde{g}_n\rangle|&\geq|\langle f_n^0,g_n^0\rangle|-\kappa_n=|\overline{A_n}f_{n,1}^0+\overline{B_n}f_{n,2}^0|-\kappa_n\\
&\geq\frac{|1-(-1)^nb|}{\sqrt{|a|^2+|1-(-1)^nb|^2}}|f_{n,1}^0|-\frac{|a|}{\sqrt{|a|^2+|1-(-1)^nb|^2}}|f_{n,2}^0|-\kappa_n\\
&\geq\frac{|1-(-1)^nb|}{\sqrt{|a|^2+|1-(-1)^nb|^2}}\cdot\frac{\sqrt{M}}{\sqrt{M+1}}-\frac{|a|}{\sqrt{|a|^2+|1-(-1)^nb|^2}}\cdot\frac{1}{\sqrt{M+1}}-\kappa_n\\
&\geq\frac{|a|}{\sqrt{|a|^2+|1-(-1)^nb|^2}.\sqrt{M+1}}-\kappa_n.
\end{align*}
So, for large enough $|n|$ we get
\begin{equation}\label{lower2}
|\langle f_n,\tilde{g}_n\rangle|\geq D_5>0,
\end{equation}
where 
$$D_5=min\bigg\{\frac{|a|}{2\sqrt{|a|^2+|1-b|^2}.\sqrt{M+1}}\ ,\ \frac{|a|}{2\sqrt{|a|^2+|1+b|^2}.\sqrt{M+1}}\bigg\}\cdot$$
Similarly, the case $(ii)$ in (\ref{subcases}) also gives the existence of a positive lower bound for $|\langle f_n,\tilde{g}_n\rangle|$.

Now, we consider the equation (\ref{ilker6}). First, we note that $t_n=-\tau_n.\ell_0(f_n)$ and $|\ell_0(f_n)|\geq|\ell_0(f_n^0)|-\kappa_n\geq C_0-\kappa_n,$ where $C_0$ is a positive number depending on general boundary conditions and $\kappa_n\leq C_0/2.$ We also have that $|\langle
f_n,\tilde{g}_n\rangle|\geq D_5>0.$ Hence, we may divide both sides of the equality (\ref{ilker6}) by $\tau_n^{-1}t_n\langle f_n,\tilde{g}_n\rangle$ and get that 
$$|\xi|=|{\frac{\tau_n^{-1}(\mu_n-\lambda_n^+)\langle
G_n,\tilde{g}_n\rangle+\tau_n^{-1}t_n\gamma_n\langle\varphi_n,\tilde{g}_n\rangle}{\tau_n^{-1}t_n\langle
f_n,\tilde{g}_n\rangle}}|.$$
Then, due to (\ref{ilker4}), (\ref{lower1}), (\ref{lower2}) we have
\begin{align*}
|\xi_n|&=|{\frac{\tau_n^{-1}(\mu_n-\lambda_n^+)\langle
G_n,\tilde{g}_n\rangle+\tau_n^{-1}t_n\gamma_n\langle\varphi_n,\tilde{g}_n\rangle}{\tau_n^{-1}t_n\langle
f_n,\tilde{g}_n\rangle}}|\\
&\leq\frac{|\mu_n-\lambda_n^+|(|C|+1)+|\gamma_n|.|\ell_0(f_n)|.|\overline{A_nf_{n,2}^0}-\overline{B_nf_{n,1}^0}+1|}{|-\ell_0(f_n)|.|\langle
f_n,\tilde{g}_n\rangle|}\\
&\leq\frac{|\mu_n-\lambda_n^+|(|C|+1)+3.|\gamma_n|.(1+|b|+|a|+1)}{\frac{C_0}{2}D_5}\cdot
\end{align*}
So, we obtain
\begin{equation}\label{haha}
|\xi_n|\leq D_6|\mu_n-\lambda_n^+|+D_7|\gamma_n|,
\end{equation}
where
$$D_6=\frac{|C|+1}{\frac{C_0}{2}D_5},\quad D_7=\frac{3.(2+|b|+|a|)}{\frac{C_0}{2}D_5}\cdot$$
We also have
\begin{equation}\label{haha2}
\frac{1}{2}(|B_n^+|+|B_n^-|)\leq\|(z_n^+-S(\lambda_n^+))P_n^0\varphi\|
\end{equation}
for sufficiently large $|n|$ by Lemma 60 in \cite{instability}. Moreover, we have
\begin{equation}\label{haha3}
\|(z_n^+-S(\lambda_n^+))P_n^0\varphi\|\leq2(|\xi_n|+|\gamma_n|)
\end{equation}
for sufficiently large enough $|n|$ by Lemma 59 in \cite{instability}.
Hence, by (\ref{haha}), (\ref{haha2}) and (\ref{haha3}) we obtain
\begin{align*}
|B_n^+|+|B_n^-|&\leq4|\xi_n|+4|\gamma_n|\leq D_8|\mu_n-\lambda_n^+|+D_9|\gamma_n|,
\end{align*}
where $D_8=4D_6$ and $D_9=4D_7+4.$ This completes the proof.
\end{proof}

In the above results, the analogues of Proposition 62, Proposition 63 and Proposition 65 in \cite{instability} (which are used in the proof of Theorem 66 in \cite{instability}) are given as Proposition \ref{prop1}, Proposition \ref{prop2} and Proposition \ref{prop3}, respectively. Hence, as in Theorem 66 in \cite{instability} we get the proof of Theorem \ref{intro55555} which claims that for $n\in\mathbb{Z}$ with large enough $|n|$ there are constants $K_1>0$ and $K_2>0$ such that
$$K_1(|\beta_n^-(z_n^*)|+|\beta_n^+(z_n^*)|)\leq|\lambda^+_n-\lambda_n^-|+|\mu_n^{bc}-\lambda_n^+|\leq K_2(|\beta_n^-(z_n^*)|+|\beta_n^+(z_n^*)|).$$

Theorem \ref{intro333} and Theorem \ref{intro4444} now follow from Theorem \ref{intro55555} due to the asymptotical equivalence of the sequence $(|\beta_n^+(z_n^*)|+|\beta_n^-(z_n^*)|)$ with each of the sequences $(\Delta_n)$ and $(\Delta_n^{bc})$.

\end{document}